\documentclass[11pt]{amsart}
\usepackage{graphicx}
\usepackage{color}\usepackage{amsmath,amsthm,amssymb,amsfonts}
\usepackage{amsmath}
\usepackage{amssymb}
\newtheorem{teo}{Theorem}[section]
\newtheorem{cor}[teo]{Corollary}
\newtheorem{lema}[teo]{Lemma}
\newtheorem{prop}[teo]{Proposition}
\theoremstyle{definition}

\theoremstyle{remark}
\newtheorem{rem}[teo]{Remark}
\newtheorem{ex}[teo]{Example}
\numberwithin{equation}{section}

\newcommand{\N}{\mathbb{N}}
\newcommand{\R}{\mathbb{R}}
\newcommand{\C}{\mathbb{C}}
\newcommand{\D}{\mathbf{D}}
\newcommand{\B}{\mathcal{B}}

\renewcommand{\phi}{\varphi}

\begin{document}

\def\N{\mathbb{N}}
\def\Z{\mathbb{Z}}
\def\K{\mathbb{K}}
\def\R{\mathbb{R}}
\def\D{\mathbb{D}}
\def\C{\mathbb{C}}
\def\T{\mathbb{T}}

\def\dint{\int}
\def\intc{\int_0^1}
\def\intD{\int_{\D}}
\def\intpi{\int_0^{2\pi}}
\def\sumi{\sum_{n=0}^\infty}
\def\dfrac{\frac}
\def\dsum{\sum}
\def\cuadro{\hfill $\Box$}
\def\qed{\hfill $\Box$}
\def\prueba{\vskip10pt\noindent{\it PROOF.}\hskip10pt}
\newcommand{\cuadrosymb}{\mbox{ }~\hfill~{\rule{2mm}{2mm}}}
\renewcommand{\phi}{\varphi}

\def\L{{\cal L}}
\def\P{{\cal P}}
\def\B{\mathcal{B}(B_E)}
\def\H{\cal H}
\def\KO{{\cal K}}

\def\HX{{\cal H}(X)}
\def\HY{{\cal H}(Y)}
\def\HZ{{\cal H}(Z)}

\def\ba{\begin{eqnarray*}}
\def\ea{\end{eqnarray*}}

\def\be{\begin{equation}}
\def\ee{\end{equation}}

\def\r{\right}
\def\d{\displaystyle}
\def\lh{{\langle}}
\def\rh{{\rangle}}
\def\Rad{\mathop{\text{\rm{Rad}}}\nolimits}
\def\id{\mathop{\hbox{\rm id}}\nolimits}
\def\rad{\mathop{\hbox{\rm rad}}\nolimits}
\def\wcj{{\bar w}}
\def\zcj{{\bar z}}
\def\nuc{\big|K_z(\wcj)\big|}
\def\tp{\hat\otimes}
\def\I{{1\over 2\pi}\int_{-\pi}^\pi}
\def\II{\int_{0}^1}

\def\e{\varepsilon}
\def\o{\over}

\def\s{\sum_{n=0}^\infty}
\def\sf{\sum_{n\ge 0}}
\def\nl{\nolimits}
\setlength{\textheight}{210mm}

\title[Composition Operators on the Bloch space of a Hilbert Space]{Composition Operators on the Bloch space of the Unit Ball of a Hilbert Space}
\author[O. Blasco]{Oscar Blasco$^*$} \address{Oscar Blasco. Departamento de An\'{a}lisis Matem\'{a}tico, Universidad de Valencia, Valencia, Spain. \emph{e}.mail:
oscar.blasco@uv.es}

\author[P. Galindo]{Pablo Galindo$^\dag$}\address{Pablo Galindo. Departamento de An\'{a}lisis Matem\'{a}tico, Universidad de Valencia, Valencia, Spain. \emph{e}.mail:
pablo.galindo@uv.es}

\author[M. Lindstr\"{o}m]{Mikael Lindstr\"om$^\dag$}\address{Mikael Lindstr\"om. Department of Mathematics, Abo Akademi University, Abo, Finland. \emph{e}.mail: mlindstr@abo.fi}

\author[A. Miralles]{Alejandro Miralles$^\ddagger$}\address{Alejandro Miralles. Departament de Matem\`{a}tiques and
Instituto Universitario de Matem\'{a}ticas y Aplicaciones de Castell\'{o}n (IMAC), Universitat Jaume I de Castell\'{o} (UJI), Castell\'{o}, Spain. \emph{e}.mail:
mirallea@uji.es}

\thanks{$^*$Partially supported by  MTM2011-23164.}\thanks{$^\dag$Partially supported by MEC2011-22457. }\thanks{$^\ddagger$Partially supported by MEC2011-22457 and P1-1B2014-35.}
\maketitle
\begin{abstract} Every analytic self-map of the unit ball of a Hilbert space induces a bounded composition operator on the space of Bloch functions. Necessary and sufficient conditions for compactness of such composition operators are provided, as well as some examples that clarify the connections among such conditions.
\end{abstract}

\section{Introduction }
Let $E$ be a complex Hilbert space of arbitrary dimension and denote $B_E$ its open unit ball. The space $\mathcal{B}(B_E)$ of Bloch functions was introduced in \cite{BGM}. There it was shown that it can be endowed with a (modulo the constant functions) norm  that is invariant under the automorphisms of $B_E;$ see  section 3 below for the basics. This article studies composition operators acting on $\mathcal{B}(B_E),$ i.e., self-maps of $\mathcal{B}(B_E)$ defined according to $C_\varphi(f)=f\circ\varphi,$  for a given analytic map $\varphi:B_E \to B_E.$ As in the finite dimensional case, every composition operator is bounded, actually of norm not greater than $1$ for the invariant norm if the symbol vanishes at $0,$ and also the hyperbolic metric on $B_E$ measures the distance between evaluations in the dual space. We also study the compactness of composition operators, providing necessary and sufficient conditions. There are two common requirements for both the necessity and the sufficiency:
 $$\lim_{\|\phi(z)\|\to 1}\frac{(1- \|z\|^2) \|\mathcal R \phi(z)\|}{\sqrt{1- \|\phi(z)\|^2}}=0\quad \hbox{and}  $$
$$\lim_{\|\phi(z)\|\to 1}\frac{(1- \|z\|^2) |\langle  \phi(z), {\mathcal R\phi(z)}\rangle|}{1- \|\phi(z)\|^2}=0, $$
\noindent where $R\phi(z)$ denotes the radial derivative at $z.$  The fact that for all $0<\delta<1,$ $\varphi(\delta B_E)$ and $ \{(1-\|z\|^2)\mathcal R\phi(z):z\in B_E\}$ are relatively compact completes a necessary condition, while the additional assumption $\varphi(B_E)\cap \delta B_E$ and $ \{(1-\|z\|^2)\mathcal R\phi(z):\|\phi(z)\|<\delta \}$ being relatively compact, provides a sufficient one. Such compactness requirements are trivially satisfied in the finite dimensional case, thus the two limits above   yield  an apparently new characterization.

Some of our techniques are inspired by J. Dai's paper \cite{D}. However, there are some obstacles to avoid when allowing an infinite number of variables, like the lack of relative compactness of the ball, the number of components of the symbol or the use of the invariant Laplacian. And still a major one: uniform convergence on compact sets does not imply uniform convergence on compact sets of the derivatives; this only happens in the finite dimensional setting (see \cite{C}). Such obstacle causes the lengthy proof of our main result Theorem \ref{main}. In the final section we present several examples that discuss the relations among the conditions we have found.
\section{Background}
Let $(e_k)_{k\in \Gamma}$ be an orthonormal basis of $E$ that we fix at once.  Then every $z\in E$ can be
 written as $z= \sum_{k\in \Gamma} z_ke_k$ and we write $\overline{z}= \sum_{k\in \Gamma} \overline{z_k}e_k.$
 \smallskip

Given an analytic function $\phi:B_E\to B_E$
we write $\phi(x)=\sum_{k \in \Gamma} \phi_k(x)e_k$,
$\phi'(x):E\to E$ its derivative at $x,$
and
$\mathcal R \phi(x)=\phi'(x)(x)$ its radial derivative at $x.$

We shall denote by $\varphi_a$ the M\"obius transforms for Hilbert spaces. For each $a\in B_E,\;\varphi_a: B_E\to B_E$ is defined by
$$ \varphi_a(x)=(s_aQ_a+P_a)(m_a(x))$$
where $s_a=\sqrt{1-\|a\|^2}$, $m_a:B_E\to B_E$ is the analytic function
$$m_a(x)=\frac{a-x}{1-\langle x,a\rangle}$$
and $P_a= \frac{1}{\|a\|^2}a\otimes a$ where $u\otimes v(x)=\langle x,u\rangle v$ and $Q_a=Id- P_a$ are the orthogonal projection on the one dimensional subspace generated by $a$ and on its orthogonal complement respectively.
Since  $\varphi_a \circ \varphi_a(x)=x$ one has $(\varphi_a)^{-1}= \varphi_a$ and $\varphi_a'(a)=(\varphi'_a(0))^{-1}.$

Actually (see for instance  \cite[Lemma 3.2]{BGM})
\begin{equation}\label{eq7}
\varphi'_a(0)=-s_a^2 P_a  -s_a Q_a,
\end{equation}

\noindent and

\begin{equation}\label{eq8}
 \varphi'_a(a)=-\frac{1}{s_a^2}P_a  -\frac{1}{s_a}Q_a.
\end{equation}
The pseudo-hyperbolic and hyperbolic metrics on $B_E$ are  respectively defined by $$\rho_E(x,y):=\|\varphi_{x}(y)\| ~~\text{    and    }~~
 \beta_E(x,y):=\frac{1}{2}\log\frac{1+\rho_E(x,y)}{1-\rho_E(x,y)}.$$ It is known (\cite{GR} p. 99)
that
\begin{equation} \label{eq2}
 \|\varphi_{x}(y)\|^2=1-\frac{(1 -\|x\|^2)(1 -\|y\|^2)}{| 1 - \langle x,y\rangle |^2}.
\end{equation}
Also
\begin{equation}\label{eq3}
\rho_E (x,y)\ =\ \sup\{\rho(f(x),f(y)): f\in H^\infty(B_E), \|f\|_\infty \le 1\},
\end{equation}
where $\rho$ is the pseudo-hyperbolic metric  on the open unit disk $\mathbb D$ in the complex plane given by
$\rho(z,w)= \left| \frac{z-w}{1-\bar{z}w}\right|$ and  $H^\infty(B_E)$ denotes the Banach space of  bounded analytic functions on $B_E$ endowed with the sup-norm.

Since $(s+t)/(1+st)$ is an increasing function of $s$ and $t$ for $0\leq s,t\leq 1$, the sharpened form of the
triangle inequality for $\rho(z,w)$ easily yields the same inequality for
$\rho_E(x,y)$,
\begin{equation}\label{eq4}
\rho_E(x,y)\ \leq \frac{\rho_E(x,u)+\rho_E(u,y)}{1+\rho_E(x,u)\rho_E(u,y)},
\qquad x,u,y\in B_E.
\end{equation}
The following estimate holds (see \cite{BGM}, Lemma 4.1):
\begin{equation} \label{eq5}
\rho_E(x,y) \le\frac{\|x-y\|}{|1- \langle x, y\rangle|}, \qquad x,y \in B_E.
\end{equation}
The open unit ball of $H^\infty(B_E)$ is invariant under post-composition with
conformal self-maps of $\mathbb D$.  By composing $f$ with a conformal
self-map of $\mathbb D$ that maps $f(y)$ to $0$, one obtains that
\be\label{eq6}
\rho_E (x,y)\ =\ \sup\{|f(x)|: f\in H^\infty(B_E), \|f\|_\infty \le 1, f(y)=0\}.
\end{equation}

Recall that if $f:B_E\to \C$ is analytic we have $f'(x)(y)=\langle y,\overline{\nabla f(x)}\rangle $ and $
(f\circ \varphi_x)'(0)(y)=\langle y,\overline{\widetilde \nabla f(x)}\rangle,$
\noindent where $\widetilde \nabla f(x)$ denotes the invariant gradient of $f$ at $x \in B_E$ given by $$\widetilde \nabla f(x)=\nabla (f \circ \phi_x)(0).$$

The following result gives an explicit formula to compute the invariant gradient. It is a  modification of Lemma 3.5 in \cite{BGM} in a form that fits our purposes.
\begin{lema} \label{lema2}  Let $f:B_E \to\mathbb C$ be an analytic function and $x\in B_E$. Then
\begin{equation}  \label{fund}
\|\widetilde \nabla f(x)\|  = \sup_{w\ne 0} \frac{|\langle \nabla f(x), {\overline w }\rangle| (1 -\|x\|^2)}
{\sqrt{(1 - \|x\|^2) \|w\|^2 + |\langle w,x \rangle|^2}}.
\end{equation}
\end{lema}

\begin{proof} For the linear functional $w\in E \mapsto \langle \varphi'_x(0)(w), \overline{\nabla f(x)} \rangle$  we have
$$\|\widetilde \nabla f(x)\| =\sup_{w\ne 0} \frac{|\langle \varphi'_x(0)(w), \overline{\nabla f(x)} \rangle|}{\|w\|}=  \sup_{w\ne 0}  \frac{|\langle \nabla f(x), \overline{\varphi'_x(0)(w)}\rangle|}{\|w\|}.$$
Now we can replace  $w$ by $\varphi_x'(0)^{-1}(w)$ in the above formula, so
$$\|\widetilde \nabla f(x)\| = \sup_{w\ne 0}  \frac{|\langle \nabla f(x), \overline{w}\rangle|}{\|\varphi'_x(0)^{-1}(w)\|}.$$
In the proof of Lemma 3.5 in \cite{BGM} it is shown that
$$\|\varphi'_x(0)^{-1}(w)\| = \frac{\sqrt{(1 - \|x\|^2) \|w\|^2 + |\langle w,x \rangle|^2}}{1 -\|x\|^2},$$
so the statement follows.
\end{proof}

Throughout the paper   $\phi:B_E \to B_E$ denotes  an analytic map and given $y\in E\setminus \{0\}$ and $w\in E$ with  $\|w\|\le 1$ we write
\begin{equation} \label{notation}
\phi_{y,w}(\lambda)=
\langle\phi(\lambda\frac{y}{\|y\|}),\overline{w}\rangle, \quad |\lambda|<1.
\end{equation}
The following version of Schwarz-Pick lemma will be needed in the sequel. The analogue of these results in several variables has been proved in \cite{BLT}.
\begin{lema} \label{lema1}  Let $\phi:B_E \to B_E$ be an analytic map  and $y \in B_E$. Then

\begin{equation}\label{sch2}|\langle \mathcal{R}\phi(y), \phi(y) \rangle|\le
\|y\| \  \|\phi(y)\| \ \frac{1-\|\phi(y)\|^2}{1-\|y\|^2},
\end{equation}

\be \label{sl2}
\frac{(1-\|y\|^2)}{\|y\|}\|\mathcal R\phi(y)\|+ \|\phi(y)\|^2|\langle  \frac{{\mathcal R \phi(y)}}{\|\mathcal R \phi(y)\|},\frac{\phi(y)}{\|\phi(y)\|}\rangle|^2\le 1.
\ee
\begin{equation}\label{sch3} \|\mathcal{R}\phi(y)\|\le 2 \ \frac{(1-\|\phi(y)\|^2)^{1/2}}{1-\|y\|^2}.
\end{equation}
\begin{equation}\text{ Furthermore if  }~~~~~\phi(0)=0, \text {then }~~~~~ \label{sch1}\|\phi(y)\|\le \|y\|.~~~~~~~~~~~~~~~~~~~~~~~~~~~~
\end{equation}
\end{lema}

\begin{proof} Let us fix $y\in B_E\setminus \{0\}$, $\phi(y)\ne 0 $ and
$w\in  E$ with $\|w\|\le 1$. We apply the  classical Schwarz lemma to $\phi_{y,w}$ and get for any $|\lambda|<1$ that
$$|\phi_{y,w}'(\lambda)|\le \frac{1-|\phi_{y,w}(\lambda)|^2}{1-|\lambda|^2}.$$
Now
if $\lambda\ne 0$ we have
$\phi'_{y,w}(\lambda)
=\frac{1}{\lambda}\langle \mathcal R
\phi(\lambda \frac{y}{\|y\|}),\overline w\rangle.$
Hence, for $\lambda=\|y\|$, it follows that
$$ |\langle \mathcal{R}\phi(y), \overline{w}\rangle| \le \|y\|\frac{1-|\langle\phi(y),\overline{w}\rangle|^2}{1-\|y\|^2}.$$
This shows  (\ref{sch2}) and (\ref{sl2}) by choosing $w= \frac{\overline{\phi(y)}}{\|\phi(y)\|}$ and $w=\frac{\overline{\mathcal R\phi(z)}}{\|\mathcal R\phi(z)\|}$ respectively.
$$\hbox{ To get (\ref{sch3}) we use the estimate } ~~|\langle \mathcal{R}\phi(y), \overline{w} \rangle|\le
2  \|y\| \ \frac{1-|\langle\phi(y),\overline{w}\rangle|}{1-\|y\|^2}.$$
In particular, for any $\theta\in [-\pi, \pi)$ and $\|w\|=1,$ we see that
$$\Big|\Big\langle \frac{(1-\|y\|^2}{2 \|y\|} \mathcal{R}\varphi(y)
 + e^{i\theta}\phi(y), \overline{w}\Big\rangle\Big|
\le \frac{1-\|y\|^2}{2 \|y\|}|\langle \mathcal{R}\phi(y),\overline{w}\rangle| + |\langle\phi(y), \overline{w}\rangle|\le 1.$$
 $$\text{  Hence,~ }~~~~ \Big\| \frac{1-\|y\|^2}{2 \|y\|} \mathcal{R}\phi(y) + e^{i\theta}\phi(y)\Big\|\le 1 ~~\text{ for } \theta\in [-\pi, \pi)~.~~~~~~~~~~~~~~$$
 $$\text{Now integrating over } \theta ~\text{we obtain }~~~~~~~\frac{(1-\|y\|^2)^2}{4\|y\|^2} \|\mathcal{R}\phi(y)\|^2 +\|\phi(y)\|^2\le 1.$$

 In the case $\phi(0)=0$ using $\phi_{y,w}(0)=0$ and scalar Schwarz lemma  we obtain
 $$|\phi_{y,w}(\lambda)|\le |\lambda|$$
 for all $y\in B_E\setminus \{0\}$, $\phi(y)\ne 0 $ and
$w\in  E$. This implies (\ref{sch1}) choosing again $\lambda= \|y\|$ and $w= \frac{\overline{\phi(y)}}{\|\phi(y)\|}$.
This completes the proof.
\end{proof}

For background on analytic (or holomorphic) mappings on infinite dimensional complex spaces we refer to \cite{C}.
\section{The Bloch space}

The classical Bloch space $\mathcal{B}$ is the space of analytic functions on the open unit disk, $f: \D \to \C$, such that the semi-norm $\|f\|_{\mathcal{B}} =\sup_{z \in \D} (1-|z|^2)|f(z)|$ is bounded; it becomes a Banach space  when endowed with the norm $\|f\|_{Bloch}=|f(0)|+\|f\|_{\mathcal{B}}.$  See \cite{Z2} for general background on the classical Bloch space. The Bloch space of functions defined on the finite dimensional Euclidean ball was introduced by R. Timoney in \cite{T80}. See \cite{Z} for further information. \medskip

 A function $f:B_E\to \C$ is a Bloch function if
$$\|f\|_{\B}=\sup_{ x\in B_E} (1-\|x\|^2)\|f'(x)\|<\infty.$$

\noindent The space of Bloch functions is denoted by $\B$ and it has been studied in \cite{BGM}. As in the finite dimensional case, the space $H^{\infty}(B_E)$  is strictly contained in $\B$ (see  \cite[Corollary 4.3]{BGM}) and the following inequality holds for any $f \in H^{\infty}(B_E)$:
\begin{eqnarray} \label{desig hinf}
\|f\|_{\B} \leq \|f\|_{\infty}.
\end{eqnarray}

\noindent An equivalent semi-norm for the space of Bloch functions is given by
$$\|f\|_{inv}=\sup_{ x\in B_E} \|\widetilde\nabla f(x)\|<\infty.$$
This semi-norm satisfies $\|f \circ \phi \|_{inv} =\|f\|_{inv}$ for any $f \in \B$ and any  automorphism $\phi$ of $B_E.$ The space $\B$ is usually endowed with the norm $\|f\|_{Bloch(B_E)}=|f(0)|+\|f\|_{inv}$ and then it becomes a Banach space.

Another equivalent semi-norm is given by
$$\|f\|_{rad}=\sup_{x \in B_E} (1-\|x\|^2) |\mathcal R f(x)|$$
where $\mathcal R f(x)=f'(x)(x)$ is the radial derivative of $f$ at $x.$

We refer to \cite[Thm 3.8]{BGM} for all the equivalences of these semi-norms. In particular, we have the following inequalities:
\be \label{eq21}
\|f\|_{\B} \leq \|f\|_{inv} \leq \left( 1+ \frac{\sqrt{31}}{2} \right) \|f\|_{\B}.
\ee
The following result extends Theorem 5.5 in \cite{Z2} to an infinite dimensional Hilbert space $E$:
\begin{teo} \label{teo met}
Let $f:B_E \to \C$ be an analytic function. Then,
$$\|f\|_{inv}=\sup \left\{ \frac{|f(x)-f(y)|}{\beta_E(x,y)} : x, y \in B_E, \ x \neq y \right\}.$$
\end{teo}
\begin{proof}
First we prove that
$$\|f\|_{inv} \geq M:=\sup \left\{ \frac{|f(x)-f(y)|}{\beta_E(x,y)} : x, y \in B_E, \ x \neq y \right\}.$$
\noindent If $\|f\|_{inv} = \infty$, then we are done. So take $f \in \B$ and $x,y \in B_E$. Then,
\begin{eqnarray*}
|f(x)-f(0)| &=& \left| \Big( \int_{0}^1 f'(x t)dt \Big)(x) \right| \leq \|x\| \Big\|\int_{0}^1 \frac{f'(xt)(1-\|xt\|^2)}{1-\|xt\|^2} dt\Big\|  \\
&\le&\|f\|_{\B} \int_{0}^1 \frac{\|x\|}{1-\|x\|^2|t|^2} dt = \|f\|_{\B} \frac{1}{2} \log \frac{1+\|x\|}{1-\|x\|}.
\end{eqnarray*}
Consider  $f \circ \varphi_y \in \B$. By the inequality above, (\ref{eq21}) and bearing in mind that $\|f \circ \varphi \|_{inv} = \|f\|_{inv}$ for any automorphism $\varphi$, we have that
$$| f \circ \varphi_y (z) - f \circ \varphi_y (0)| \leq  \|f \circ \varphi_y \|_{inv} \frac{1}{2} \log \frac{1+\|z\|}{1-\|z\|}= \|f \|_{inv} \frac{1}{2} \log \frac{1+\|z\|}{1-\|z\|}.$$
\noindent Selecting $z=\varphi_y(x)$, we have
\begin{eqnarray*}
| f(x) - f (y)| &\leq& \|f \circ \varphi_y \|_{inv} \frac{1}{2} \log \frac{1+\|\varphi_y(x)\|}{1-\|\varphi_y(x)\|}\\&=&
 \|f  \|_{inv} \frac{1}{2} \log \frac{1+\rho_E(x,y)}{1-\rho_E(x,y)}=\|f\|_{inv} \ \beta_E(x,y).
\end{eqnarray*}
\noindent Hence
$\|f\|_{inv} \geq M.$

 Now we prove that $\|f\|_{inv} \leq M$. 
Notice that
$$|f(x)-f(0)| \leq M \beta_E(x,0) = \frac{M}{2} \log \frac{1+\|x\|}{1-\|x\|},$$
\noindent so
$$\frac{|f(x)-f(0)|}{\|x\|} \leq \frac{M}{2\|x\|} \log \frac{1+\|x\|}{1-\|x\|}$$
\noindent for all $x \in B_E \setminus \{ 0 \}.$ For a unit vector $u \in E$, we consider the directional derivative $D_u f(0)$ given by
$$D_u f(0) = \lim_{t \rightarrow 0}\frac{f(0+tu)-f(0)}{t} = \nabla f(0) (u).$$
If $x =tu$ and taking limits when $\|x\| \rightarrow 0$, we have that
$$|\nabla f(0)(u)| \leq M \lim_{\|x\| \rightarrow 0} \frac{1}{2\|x\|} \log \frac{1+\|x\|}{1-\|x\|} = M$$
since
$\lim_{r \rightarrow 0} \frac{1}{r} \log \frac{1+r}{1-r}=2,$
 so $\|\nabla f(0)\| \leq M$. Notice that for any automorphism $\varphi$ on $B_E$, it is clear that
$$M=\sup \left\{ \frac{|(f \circ \varphi)(x)-(f \circ \varphi)(y)|}{\beta_E(x,y)} : x, y \in B_E, \ x \neq y \right\}$$
\noindent since $\beta_E(\varphi(x),\varphi(y))=\beta(x,y)$. Hence for any $x \in B_E$ we have $$\|f\|_{inv} = \sup_{x \in B_E} \|\nabla (f \circ \varphi_x )(0) \| \leq M$$ and we are done.
\end{proof}

\begin{cor} \label{ev bound} If $\delta_x(f)=f(x),$ we have that
$\delta_x \in \mathcal B(B_E)^*$ and  $\|\delta_x\| \le L_x$ where
$$L_x=\max \left\{\frac{1}{2} \log\frac{1+\|x\|}{1-\|x\|} \ , \ 1 \right\}.$$
\end{cor}
\begin{proof} From Theorem \ref{teo met}, we have for any $x \in B_E$
\begin{equation}\label{ineq}
|f(x)-f(0)| \le \frac{1}{2} \|f\|_{\B} \log\frac{1+\|x\|}{1-\|x\|}.
\end{equation}
$$\hbox{ Also }~|\delta_x(f)|  \leq |f(x)-f(0)|+|f(0)| \leq \frac{1}{2} \|f\|_{\B} \log\frac{1+\|x\|}{1-\|x\|} +|f(0)| $$ $$\leq
\max \left\{\frac{1}{2} \log\frac{1+\|x\|}{1-\|x\|} \ , \ 1 \right\} \big(\|f\|_{\B}+|f(0)|\big) = L_x \|f\|_{Bloch(B_E)}.
$$\end{proof}
\begin{rem} For $x,y\in B_E$ we have
\begin{equation}\label{otrades} \frac{1}{2}\|x-y\|\le \rho_E(x,y)\le \|\delta_x-\delta_y\|\le \beta_E(x,y).\end{equation}
In particular, we observe  that the norm topology of $\mathcal B(B_E)$ is finer than the compact open topology $co.$ \medskip
\end{rem}
As consequence of Theorem \ref{teo met}, we have that
\begin{cor}\label{cor}
An analytic function $f:B_E \to \C$ belongs to $\B$ if and only if there exists a constant $C >0$ such that
$$|f(x)-f(y)| \leq C \beta_E(x,y).$$
\end{cor}
Notice that the metric $\beta_E(x,y)$ can be also recovered from the Bloch semi-norm $\|f\|_{inv}$:
\begin{cor}
For any $x,y \in B_E$ we have
$$\beta_E(x,y)=\sup\{ |f(x)-f(y)| : \|f\|_{inv} \leq 1 \}.$$
\end{cor}

\begin{proof}
By Theorem \ref{teo met} we have that $|f(x)-f(y)| \leq \|f\|_{inv} \ \beta_E(x,y)$ for all $f \in \B$ and $x,y \in B_E.$ Hence, $\sup\{ |f(x)-f(y) | : \|f\|_{inv} \leq 1\} \leq \beta_E(x,y).$

To check the other inequality follow the same pattern as Theorem 3.9 in \cite{Z} and recall \cite[Lemma 3.3]{BGM}.
\end{proof}
\section{Composition operators}
\subsection{Boundedness} As it occurs in the finite dimensional case, every composition operator on $\mathcal B(B_E)$ is bounded.
\begin{teo}\label{cont} Every  analytic map $\phi:B_E\to B_E$  induces a bounded  composition operator $C_\phi:\mathcal B(B_E)\to \mathcal B(B_E)$. 
\end{teo}

\begin{proof}
Let $\phi:B_E \to B_E$ be analytic and consider for any $f \in \B$, the semi-norm $\|f \circ \phi\|_{inv}$. By Theorem \ref{teo met}, we have that
\begin{eqnarray*}
\|f \circ \phi \|_{inv}&=&\sup \left\{ \frac{|(f \circ \phi)(x)-(f\circ \phi)(y)|}{\beta_E(x,y)} : x, y \in B_E, \ x \neq y \right\} \\&\leq &
\sup \left\{ \frac{|(f (\phi(x))-(f( \phi(y))|}{\beta_E(\phi(x),\phi(y))} : x, y \in B_E, \ \phi(x) \neq \phi(y) \right\},
\end{eqnarray*}
\noindent where last inequality holds because $\rho_E(x,y)$ is contractive for analytic maps $\phi:B_E \to B_E$ and $h(t)=\frac{1}{2} \log \frac{1+t}{1-t}$ is non-decreasing. Since $\phi(B_E) \subset B_E$, we get the estimate
$$\|f \circ \phi \|_{inv} \leq \sup \left\{ \frac{|f(x)-f(y)|}{\beta_E(x,y)} : x, y \in B_E, \ x \neq y \right\}=\|f\|_{inv}.$$
Further, using Corollary \ref{ev bound},
$$\|C_{\phi}(f)\|_{Bloch(B_E)}=\|f \circ \phi \|_{inv} +|f(\phi(0))| \leq \|f\|_{inv}+L_{\phi(0)}\|f\|_{Bloch(B_E)} \leq$$
$$\|f\|_{inv}+|f(0)|+L_{\phi(0)}\|f\|_{Bloch(B_E)} = (1+L_{\phi(0)}) \|f\|_{Bloch(B_E)},$$

\noindent and we conclude that $C_\phi$ is bounded. 
\end{proof} \medskip

We provide another proof that relies on magnitudes that will appear further on.\smallskip

\begin{proof}  
Let $\|f\|_{inv}\le 1.$ Since
$\mathcal R (f\circ \phi)(z)=\langle \nabla f(\phi(z)),\overline{\mathcal R\phi(z)}\rangle,$
we use Lemma \ref{lema2} and obtain
$$|\mathcal R (f\circ \phi)(z)|^2 \le  \|\widetilde \nabla  f(\phi(z)\|^2 \frac{(1 - \|\phi(z)\|^2) \|\mathcal R\phi(z)\|^2 + |\langle {\mathcal R\phi(z)},\phi(z)\rangle|^2}{( 1 - \|\phi(z)\|^2)^2}.$$
By combining this with Lemma \ref{lema1} we conclude that
$$|\mathcal R (f\circ \phi)(z)|(1 -\|z\|^2) \le \sqrt{5}.$$ Thus the boundedness of $C_\varphi$ is immediate if we assume $\varphi(0)=0.$

If $\phi(0) =x\ne 0$, then we consider the mapping $\psi= \varphi_x \circ \phi,$ for which $\psi(0)=0,$ and the bounded operator $C_\psi.$  Since $\|f\circ\varphi_x\|_{inv} =
\|f\|_{inv} $ it follows, using Corollary \ref{ev bound} as well, that  $C_{\varphi_x}$ is continuous.
Hence $C_\varphi= C_\psi \circ C_{\varphi_x}$ is continuous.
\end{proof}
\begin{rem}
It is clear that if $\phi(0)=0$, then $\|C_{\phi}\| \leq 1$.
\end{rem}
\subsection{Compactness} Now we proceed to discuss necessary and sufficient conditions for a composition operator on $\mathcal B(B_E)$ to be compact. We begin with some necessary ones.
\subsubsection{Necessary conditions}
The following result is a little improvement of a result due to Dai \cite{D} for finitely many  variables.

\begin{lema} \label{helplema}  For each $z\in B_E$ with $\varphi(z)\neq 0,$ there is $\eta(z) \in E$,
$\|\eta(z)\|=1$ with $\langle \phi(z),\eta(z)\rangle=0$ such that for $\xi =\phi(z) + \sqrt{1 - \|\phi(z)\|^2}\eta(z)$ one has
$$|\langle\mathcal R\phi(z),\xi\rangle|\ge\sqrt{1 - \|\phi(z)\|^2} \|\mathcal R\phi(z)\| -\Big (1 + \frac{\sqrt{1 - \|\phi(z)\|^2}}
{\|\phi(z)\|}\Big) |\langle \mathcal R\phi(z),\phi(z)\rangle|.$$
\end{lema}
\begin{proof} We use the projection theorem for Hilbert spaces, so for each $z \in B_E$ with $\phi(z) \ne 0$
there is $\eta(z) \in E$, $\|\eta(z)\|=1$ with $\langle \phi(z),\eta(z)\rangle=0$ such that
$$\mathcal R\phi(z) = \alpha \frac{\phi(z)}{\|\phi(z)\|} + \beta \eta(z),$$
where $\alpha =\frac{\langle \mathcal R\phi(z),\phi(z)\rangle}{\|\phi(z)\|}$ and
$\beta=\langle\mathcal R\phi(z),\eta(z)\rangle.$  Clearly $\|\xi\|=1$, $\langle \phi(z),\xi\rangle = \|\phi(z)\|^2$
and $\langle\mathcal R\phi(z),\xi\rangle= \langle\mathcal R\phi(z),\phi(z)\rangle +\sqrt{1 - \|\phi(z)\|^2}\beta.$
Moreover, $|\alpha|^2 + |\beta|^2  = \|\mathcal R\phi(z)\|^2,$ so
\begin{eqnarray*}|\langle\mathcal R\phi(z),\xi\rangle|&\ge&\sqrt{1 - \|\phi(z)\|^2} |\beta| -
|\langle \mathcal R\phi(z),\phi(z)\rangle|\\
&\ge& \sqrt{1 - \|\phi(z)\|^2} |(|\mathcal{R}\phi(z)\| - |\alpha|)  -|\langle \mathcal R\phi(z),\phi(z)\rangle|\\
&=&\sqrt{1 - \|\phi(z)\|^2} \|\mathcal R\phi(z)\| - \Big(1 + \frac{\sqrt{1 - \|\phi(z)\|^2}}
{\|\phi(z)\|}\Big) |\langle \mathcal R\phi(z),\phi(z)\rangle|.
\end{eqnarray*}
 \end{proof}
\begin{lema} \label{lema3}  The composition operator $C_\phi:\mathcal B(B_E)\to \mathcal B(B_E)$ is  compact  if and only if for each bounded net $(f_\alpha)$  in  $\mathcal B(B_E)$ such that $f_\alpha \to 0$ in $(\mathcal B(B_E), co)$ it follows   that $\|C_\phi(f_\alpha)\|_{\B} \to 0.$
\end{lema}
\begin{proof}
Suppose that $C_\phi:\mathcal B(B_E)\to \mathcal B(B_E)$ is compact
and let $(f_\alpha)$ be a bounded net in  $\mathcal B(B_E)$ such that $f_\alpha \to 0$ in
$(\mathcal B(B_E), co)$. Then also $C_\phi(f_\alpha)\to 0$ in
$(\mathcal B(B_E),co)$ and the norm closure of the set $\{C_\varphi(f_\alpha),  0\}$ is compact in $\mathcal B(B_E)$.  Therefore $\|C_\phi(f_\alpha)\|_{\B} \to 0.$

If $C_\phi$ is non-compact, then there are $\varepsilon > 0$ and a sequence $(f_n)$ in $\mathcal B(B_E)$ such that $\|f_n\|_{\B} =1$ and
$$\|C_\phi(f_n) - C_\phi(f_m)\|_{\B}\ge \varepsilon \ \text{for each} \  n\ne m.$$ Now by Montel's theorem (see \cite[Theorem 17.21]{C}), there is a subnet $(f_{n(\alpha)})$
of $(f_n)$  that converges uniformly on compact subsets  of $B_E$ in $H(B_E)$. For each $n(\alpha)$, choose $n(\beta) > n(\alpha)$ such that
$f_{n(\alpha)}\ne f_{n(\beta)}$ and let $g_{n(\alpha)} = f_{n(\alpha)}- f_{n(\beta)}.$ Then $g_{n(\alpha)} \to 0$ in $(\mathcal B(B_E), co)$, but
$\|C_\phi(g_{n(\alpha)})\|_{\B} \ge \varepsilon > 0.$
\end{proof}
\begin{teo} \label{sufi1} Assume that $C_\phi:\mathcal B(B_E)\to
\mathcal B(B_E)$ is a compact operator. Then

\begin{equation} \label{c0} \phi(\delta B_E) \hbox{ is relatively compact for each } 0<\delta<1,  \end{equation}
\begin{equation} \label{c1}\lim_{\|\phi(z)\|\to 1}\frac{(1- \|z\|^2) \|\mathcal R \phi(z)\|}{\sqrt{1- \|\phi(z)\|^2}}=0, \hbox{ and }  \end{equation}

 \begin{equation} \label{c2}\lim_{\|\phi(z)\|\to 1}\frac{(1- \|z\|^2) |\langle  \phi(z), {\mathcal R\phi(z)}\rangle|}{1- \|\phi(z)\|^2}=0. \end{equation}\end{teo}
\begin{proof} 
First we prove $(\ref{c0})$: Indeed, since the set $\{\delta_z: \|z\|\le \delta\}\subset \big(\mathcal{B}(B_E)\big)^*$ is bounded and
 $C_\phi^*$ is compact,
$\{C_\phi^*(\delta_z): \|z\|\le \delta\}$ is relatively compact in
$\mathcal B(B_E)^*.$ The fact that $C_\phi^*(\delta_z)=\delta_{\varphi(z)}$ allows us
to conclude that
$\phi(\delta B_E)$  is relatively compact by appealing to $(\ref{otrades}).$

  Let $(n_k)$ be an increasing sequence in $\mathbb N$ and $(\xi_k)$ a sequence in $E$ with $\|\xi_k\|\le 1$. According to \cite[Corollary 4.3]{BGM} the family $\{\langle z,\xi\rangle^{n_k}: \|\xi\|=1\}$ is bounded in $\mathcal B(B_E).$
Furthermore the resulting sequence $\{\langle z,\xi_k\rangle^{n_k}\}$  converges to zero in $(\mathcal B(B_E),co)$ and therefore the compactness of $C_\phi:\mathcal B(B_E)\to
\mathcal B(B_E)$ implies, according to Lemma \ref{lema3}, that
\begin{equation} \label{conv}\lim_k\|\langle \phi,\xi_k\rangle^{n_k}\|_{rad}\to 0, \ \text{when} \
k\to\infty.\end{equation}
$$\text{ We have  }~~ \|\langle \phi,\xi_k\rangle^{n_k}\|_{rad}
= \sup_{z\in B_E}(1- \|z\|^2)n_k|\langle  \phi(z), \xi_k\rangle|^{{n_k}-1}|\langle {\mathcal R\phi(z)}, \xi_k\rangle|.$$

Let us first show (\ref{c2}).
We suppose that there exist $\e >0$ and a sequence $(z_k)\in B_E$ such that
$\|\phi(z_k)\|\to 1$ and for each $k$,
\be\label{cond1}
\frac{1- \|z_k\|^2}{ 1- \|\phi(z_k)\|^2}|\langle  \phi(z_k), {\mathcal R\phi(z_k)}\rangle|\ge \e
\end{equation}

Let $n_k$ be the integer part of $\frac{1}{1 - \|\phi(z_k)\|}$
and choose $\xi_k= \frac{\phi(z_k)}{\|\phi(z_k)\|}.$
Since $\lim_k (1-\|\phi(z_k)\|)n_k=1$ and $\lim_k \|\phi(z_k)\|^{{n_k}-2}=\frac{1}{e},$ it follows from (\ref{conv}) that
$$0=\lim_{k\to\infty}\frac{1- \|z_k\|^2}{ 1- \|\phi(z_k)\|^2}\|\phi(z_k)\|^{{n_k}-2}
|\langle {\mathcal R\phi(z_k)}, \phi(z_k)\rangle|$$
$$= \frac{1}{e} \lim_{k\to\infty}\frac{1- \|z_k\|^2}{ 1- \|\phi(z_k)\|^2}
|\langle {\mathcal R\phi(z_k)}, \phi(z_k)\rangle|,$$
which gives a contradiction if (\ref{cond1}) holds. Thus (\ref{c2}) holds.
\smallskip

Let us now show (\ref{c1}). As above we suppose that there exist $\e >0$ and a sequence $(z_k)\in B_E$ such that
\be\label{cond2}
 \frac{1- \|z_k\|^2}{\sqrt{ 1- \|\phi(z_k)\|^2}}  \|{\mathcal R\phi(z_k)}\|\ge \e.
\end{equation}
 Choosing now $n_k$ the integer part of $\frac{1}{1 - \|\phi(z_k)\|}$ and $\xi_k =\phi(z_k) + \sqrt{1 - \|\phi(z_k)\|^2}\eta(z_k)$ with  $\|\eta(z_k)\|=1$ and
$\langle \phi(z_k),\eta(z_k)\rangle =0,$ then we obtain  from (\ref{conv})
$$0=\lim_{k\to\infty}\frac{1- \|z_k\|^2}{ 1- \|\phi(z_k)\|^2}(\|\phi(z_k)\|^{n_k -1})^2|\langle {\mathcal R\phi(z_k)}, \xi_k\rangle|$$
$$= \frac{1}{e^2} \lim_{k\to\infty}\frac{1- \|z_k\|^2}{ 1- \|\phi(z_k)\|^2}
|\langle {\mathcal R\phi(z_k)}, \xi_k\rangle|.$$
This together with condition $(\ref{c2})$ and Lemma \ref{helplema} yields a contradiction to (\ref{cond2}). So $(\ref{c1})$ holds. \end{proof}
\begin{rem} Realize that conditions $(\ref{c1})$ and $(\ref{c2})$ hold trivially true in case $\phi(B_E)\subset r B_E$ for some $0\le r<1.$ \end{rem}

\begin{rem} Note that $\phi(z)=z$ satisfies $(\ref{c1})$ and fails $(\ref{c2})$.
Also observe that $$
\frac{(1-\|z\|^2)\langle \mathcal R\phi(z), \phi(z)\rangle}{1- \|\phi(z)\|^2}=
\frac{(1-\|z\|^2)\|\mathcal R \phi(z)\|}{\sqrt{1-\|\phi(z)\|^2}} \frac{\langle \frac{\mathcal R\phi(z)}{\|\mathcal R\phi(z)\|}, \phi(z)\rangle}{\sqrt{1-\|\phi(z)\|^2}}.
$$
Hence $(\ref{c2})$ implies $(\ref{c1})$ if there exists $\delta>0$ such that $$\inf_{\|\phi(z)\|\ge \delta}\frac{|\langle \frac{\mathcal R\phi(z)}{\|\mathcal R\phi(z)\|}, \phi(z)\rangle|}{\sqrt{1-\|\phi(z)\|^2}}>0.$$
\end{rem}

\begin{prop} \label{sufi3} Let $\phi:B_E\to B_E$ be analytic such that $C_\phi:\mathcal B(B_E)\to
\mathcal B(B_E)$ is a compact operator. Then
$\{\mathcal R\phi(z): \|z\|\le \delta\}$ is relatively compact for all $\delta<1$ as well as $\{(1-\|z\|^2)\mathcal R\phi(z):z\in B_E\}$.
\end{prop}
\begin{proof}
For $z\in B_E \text{ and } w\in E$ we consider the linear functional  $\lambda_{z,w}$ acting on $f\in \mathcal{B}(B_E)$ according to $\lambda_{z,w}(f)=f'(z)(w)=\langle w, \overline{\nabla f(z)}\rangle.$ It is continuous since $|\lambda_{z,w}(f)|\le \frac{\|w\|}{1-\|z\|^2}\|f\|_{\B}.$ Realize that $$C_\phi^*(\lambda_{z,w})(f)=\lambda_{z,w}(f\circ \phi)=\langle \phi'(z)w, \overline{\nabla f(\phi(z))}\rangle,$$ and thus that $C_\phi^*(\lambda_{z,z})=\lambda_{\phi (z), \mathcal{R}\phi (z)}.$
\smallskip

Notice that
$\mathcal R\phi(\delta B_E)$ is a bounded subset of $E$ by (\ref{sch3}) in Lemma \ref{lema1}.

Since $C_\phi^*$ is compact and $\sup\{\| \lambda_{z,z}\|:\|z\|\le \delta\}<\infty$ then
$$\{C_\phi^*(\lambda_{z,z}): \|z\|\le \delta\}= \{\lambda_{\phi (z), \mathcal{R}\phi (z)}: \|z\|\le \delta\}$$ is relatively compact in
$\mathcal B(B_E)^*$. Now we conclude that
$\mathcal R\phi(\delta B_E)$ is relatively compact because for the function $e_u(z)= \langle z, u\rangle,$ we have
$\mathcal R C_\phi(e_u)(z)=\langle\mathcal{R}\phi (z),u\rangle=\lambda_{\phi (z), \mathcal{R}\phi (z)}(e_u)$ and  hence
$$\|\mathcal R\phi(z)-\mathcal R\phi(z')\|=\sup_{\|u\|\leq 1} |\langle \mathcal R\phi(z)-\mathcal R\phi(z'),u\rangle |\le \|\lambda_{\phi(z), \mathcal{R}\phi (z)} -\lambda_{\phi(z'), \mathcal{R}\phi (z')}\|.$$

Moreover, $\{(1-\|z\|^2)\lambda_{z,w}:z,w\in B_E\}$ is also a bounded set in
$\mathcal B(B_E)^*$ and thus $$\{C_\phi^*((1-\|z\|^2)\lambda_{z,z}): \|z\|<1\}=
\{(1-\|z\|^2)\lambda_{\phi (z), \mathcal{R}\phi (z)}:\|z\|<1\}$$ is a relatively compact set. So the compactness of $\{(1-\|z\|^2)\mathcal R\phi(z):z\in B_E\}$ follows as above.
\end{proof}

There are also necessary conditions in terms of the components of the symbol $\varphi.$
Recall that $(e_k)_{k\in \Gamma}$ is an orthonormal basis of $E$ and $\phi=\sum_{k \in \Gamma} \phi_k(x) e_k$. Here, $\varphi_k= \langle \varphi, e_k \rangle.$
\begin{prop} \label{sufi2} Assume that $C_\phi:\mathcal B(B_E)\to \mathcal B(B_E)$ is a compact operator. Then
\begin{equation} \label{c4} C_{\phi_{k,l}}: \mathcal B \to \mathcal B \hbox{ is compact }
\end{equation}
 for all $k,l\in \Gamma,$ where  $\phi_{k,l}(\lambda):=\phi_k
(\lambda e_l),\; \lambda\in \D.$ Also,
 \begin{equation} \label{c3} \lim_{k\in \Gamma} \sup_{z\in B_E} \frac{(1-\|z\|^2)|\mathcal R\phi_k(z)|}{1-|\phi_k(z)|^2}=0.
\end{equation}
In particular, $\lim_{k\in \Gamma} \|\phi_k\|_{\mathcal B(B_E)}=0.  \text{  And further,  }$
\begin{equation} \label{c3'} \lim_{|\phi_n(z)|\to 1}  \frac{(1-\|z\|^2)|\mathcal R\phi_n(z)|}{1-|\phi_n(z)|^2}=0 \quad n\in \Gamma,
\end{equation}

\end{prop}
\begin{proof} Let $y\in E\setminus \{0\}$ and $\|\xi\|\le 1.$ We write $F^\xi(x)=F(\langle x,\bar \xi\rangle),\;x\in B_E,$ for each $F\in \mathcal H(\D),$ and $f_y(\lambda)= f(\lambda \frac{y}{\|y\|}),\;\lambda\in \D,$ for each $f \in \mathcal H(B_E)$.

Consider $F\in \mathcal B$. Since
$\nabla F^\xi(x)=F'(\langle x,\bar \xi\rangle)\xi$ then  $F^\xi\in \mathcal B(B_E)$ and
$$(1-\|x\|^2)\|\nabla F^\xi(x)\|\le \|\xi\|\|F\|_{\mathcal B}\frac{1-\|x\|^2}{1- |\langle x, \bar\xi\rangle|^2}\le \|F\|_{\mathcal B}.$$ Hence the operator $E_\xi:F\in\mathcal{B}\mapsto F^\xi\in \mathcal{B}(B_E)$ is continuous.

If $f\in \mathcal B(B_E)$ and $\|y\|\le 1$ then it is an easy calculation that $f_y\in \mathcal B$ and $\|f_y\|_{\mathcal B}\le \|f\|_{\mathcal B(B_E)}.$ Hence the operator
$R_y:f\in\mathcal{B}(B_E)\mapsto f_y \in\mathcal{B}$ is continuous.
For each $y, \xi\in B_E$ and  $F\in \mathcal B$ we can write
$$ (C_\phi(F^\xi))_y(\lambda)= F^\xi\left(\phi \left(\lambda\frac{y}{\|y\|}\right)\right)= F\left(\left\langle\phi\big(\lambda\frac{y}{\|y\|}\big),\bar\xi\right\rangle\right)= C_{\phi_{y,\xi}}(F)(\lambda).$$
\noindent So $C_{\phi_{y,\xi}}=R_y\circ C_\phi \circ E_\xi$
 is compact.
Then (\ref{c4}) follows
because $\phi_{k,l}=\phi_{e_k,e_l}$.
\smallskip

Let us now show  (\ref{c3}).
  Given a weakly null net $(\xi_k)_{k\in \kappa}\in E$ with $\|\xi_k\|\le 1,$  we consider $f_k(z)= \log\big(\frac{1}{1-\langle z,\xi_k\rangle }\big).$  According to \cite[Corollary 4.4]{BGM}, $f_k\in \mathcal B(B_E)$ and
$\|f_k\|_{\mathcal B(B_E)}\le \|\log(\frac{1}{1-\lambda})\|_\mathcal{B}.$ Thus, the net $\{f_k:k\in \kappa\}$ is  bounded on compact subsets in $B_E,$ hence a $co$-relatively compact set by Montel's theorem. Since $\lim_{k\in \kappa}f_k(z)=0,$ it follows that  $\{f_k:k\in \kappa\}$ converges to zero uniformly on compact sets of $B_E$. Hence $\lim_{k\in \kappa}\|C_\phi(f_k)\|_{\mathcal B(B_E)}=0.$
Now notice that $\mathcal R(C_\phi(f_k))(z)= \frac{\langle\mathcal R \phi(z),\xi_k\rangle}{1-\langle\phi(z),\xi_k\rangle}.$
Therefore \begin{equation}\label{eq0}
\lim_{k\in \kappa}\sup_{\|z\|< 1}  \frac{(1-\|z\|^2)|\langle\mathcal R \phi(z),\xi_k\rangle |}{|1-\langle\phi(z),\xi_k\rangle|}=0.
\end{equation}

Assume now that (\ref{c3}) does not hold. Then there exist $\e>0$,  and a subnet $(n_k)$ such that for every $n_k$ there is $z_k$ with
\begin{equation}\label{eq2'}\frac{(1-\|z_k\|^2)|\mathcal R\phi_{n_k}(z_k)|}{1-|\phi_{n_k}(z_k)|^2}\ge \e.\end{equation}

Selecting now $\xi_k=e_{n_k} \overline{\phi_{n_k}(z_k)},$ we get  a weakly null net for which thus (\ref{eq0}) holds. Then
 $$\sup_{\|z\|< 1}  \frac{(1-\|z\|^2)|\mathcal R \phi_{n_k}(z)\|\phi_{n_k}(z_k) |}{|1-\phi_{n_k}(z)\overline{\phi_{n_k}(z_k)}|}\to 0, \quad k\to \infty,$$
that contradicts (\ref{eq2'}).

Finally, we prove (\ref{c3'}).
Let $n\in \Gamma$ and assume that (\ref{c3'}) does not hold, that is  there is $\e>0$ and a sequence $(z_l)$ with $\lim_{l\to\infty}|\phi_n(z_l)|= 1$  and \begin{equation}\label{eq1}\frac{(1-\|z_l\|^2)|\mathcal R\phi_n(z_l)|}{1-|\phi_n(z_l)|^2}\ge \e.\end{equation}

Let $F_l(\lambda)=\log\frac{1}{1-\lambda \overline{\varphi_n(z_l)}}$ and $g_l(x)=F_l(\langle x, e_n \rangle)=\log\frac{1}{1-\langle  x,e_n \rangle \overline{\varphi_n(z_l)}}.$

We may assume that $\varphi_n(z_l)$ converges to some $w_0,\;|w_0|=1.$ This means that $(g_l)$ $co$-converges to $g_0(x)=F_0(\langle x, e_n \rangle)=\log\frac{1}{1-\langle  x,e_n \rangle \overline{w_0}}$ where $F_0(\lambda)=\log\frac{1}{1-\lambda \overline{w_0}}.$
Next, notice that $C_\varphi(g_l)(x)=F_l\big( \langle  \varphi(x),e_n\rangle\big)=F_l\circ \varphi_n(x).$
\smallskip

 The compactness of $C_\varphi$ yields that  $\lim_l \|C_\varphi(g_l)-C_\varphi(g_0)\|_{rad}=0.$ However, $$\|C_\varphi(g_l)-C_\varphi(g_0)\|_{rad}=\|F_l\circ \varphi_n- F_0\circ \varphi_n\|_{rad}=$$ $$\sup_{x\in B_E}(1-\|x\|^2)|\mathcal{R}(F_l\circ \varphi_n)(x)-\mathcal{R}(F_0\circ \varphi_n)(x)|=$$
 $$\sup_{x\in B_E}(1-\|x\|^2)||F'_l(\varphi_n(x))\mathcal{R}\varphi_n(x)-F'_0(\varphi_n(x))\mathcal{R}\varphi_n(x)| =$$ $$\sup_{x\in B_E}(1-\|x\|^2)|\mathcal{R}\varphi_n(x)| \big|F'_l(\varphi_n(x))-F'_0(\varphi_n(x))\big|=$$ $$\sup_{x\in B_E}(1-\|x\|^2)|\mathcal{R}\varphi_n(x)|\Big|\frac{\overline{\varphi_n(z_l)}}{1-\overline{\varphi_n(z_l)}\varphi_n(x)}
 -\frac{\overline{w_0}}{1-\overline{w_0}\varphi_n(x)}\Big| \ge $$

 $$(1-\|z_l\|^2)|\mathcal{R}\varphi_n(z_l)|\Big|
 \frac{\overline{\varphi_n(z_l)}}{1-\overline{\varphi_n(z_l)}\varphi_n(z_l)}
 -\frac{\overline{w_0}}{1-\overline{w_0}\varphi_n(z_l)}\Big|=$$ $$
 \frac{(1-\|z_l\|^2)}{1-|\varphi_n(z_l)|^2}|\mathcal{R}\varphi_n(z_l)
 \Big|\frac{\overline{\varphi_n(z_l)}-\overline{w_0}}
 {1-\varphi_n(z_l)\overline{w_0}}\Big|\ge\e. $$
A contradiction.
\end{proof}

%
%


\subsubsection{Compactness criteria}

\begin{lema} \label{lema tec}
 Let $f:B_E\to \C$ be analytic and $x\in B_E$. Then
%
\begin{equation}\label{eq10}
(1-\|x\|^2)\mathcal Rf(x)=\frac{-1}{2\pi i }\int_{|\xi|=1} f(\varphi_x(\xi x))\frac{d\xi}{\xi^2}\Big .
\end{equation}
\end{lema}
\begin{proof}
 Observe that since $\varphi_x$ is self-inverse, $f=(f\circ \varphi_x)\circ \varphi_x,$  hence for
 $y\in B_E$
 \begin{eqnarray*}\langle y, \overline{\nabla f(x)}\rangle &=&f'(x)(y)=(f\circ \varphi_x)'(0)\circ (\varphi_x)'(x)(y)\\&=&
 (f\circ \varphi_x)'(0)\big[(-\frac{1}{s_x^2}P_x-\frac{1}{s_x}Q_x)(y)\big]\\
 &=&-\frac{1}{s_x^2}(f\circ \varphi_x)'(0)[P_x(y)]-\frac{1}{s_x}(f\circ \varphi_x)'(0)[Q_x(y)]\\&=&
 -\frac{1}{s_x^2}\langle P_x(y),\overline{\widetilde \nabla f(x)}\rangle  -\frac{1}{s_x}\langle Q_x(y), \overline{\widetilde \nabla f(x)}\rangle\\
 &=&
 -\frac{1}{s_x^2}\langle P_x(y),\overline{\widetilde \nabla f(x)}\rangle -\frac{1}{s_x}\langle y- P_x(y), \overline{\widetilde \nabla f(x)}\rangle.\end{eqnarray*}
\noindent and using that $P_x$ is self-adjoint,
\begin{eqnarray*}\langle y, \overline{\nabla f(x)}\rangle &=&-\frac{1}{s_x^2}\langle y,P_x\big(\overline{\widetilde \nabla f(x)}\big)\rangle -\frac{1}{s_x}\langle y, \overline{\widetilde \nabla f(x)}\rangle +\frac{1}{s_x}\langle y, P_x\Big(\overline{\widetilde \nabla f(x)}\Big)\rangle\\
&=& (-\frac{1}{s_x^2}+\frac{1}{s_x})\langle y, \frac{\langle \overline{\widetilde \nabla f(x)},x\rangle }{\|x\|^2} x\rangle -\frac{1}{s_x}\langle y, \overline{\widetilde \nabla f(x)}\rangle\\
 &=&(-\frac{1}{s_x^2}+\frac{1}{s_x})\frac{\langle x,\overline{\widetilde \nabla f(x)}\rangle }{\|x\|^2}\langle y,x\rangle -\frac{1}{s_x}\langle y, \overline{\widetilde \nabla f(x)}\rangle .\end{eqnarray*}
\begin{equation} \text { Therefore, }~~\label{eq11} s_x^2\langle y, \overline{\nabla f(x)}\rangle =(s_x-1)\frac{\langle x,\overline{\widetilde \nabla f(x)}\rangle }{\|x\|^2}\langle y,x\rangle -s_x\langle y, \overline{\widetilde \nabla f(x)}\rangle .\end{equation}

By the Cauchy formula we have that $$\langle x,\overline{\widetilde \nabla f(x)}\rangle= (f\circ \varphi_x)'(0)(x)=\frac{1}{2\pi i}\int_{|\xi|=1}f\circ \varphi_x(\xi x)\frac{d\xi}{\xi^2}~~~~ \text { and } $$
$$\langle y,\overline{\widetilde \nabla f(x)}\rangle= (f\circ \varphi_x)'(0)(y)=\frac{1}{2\pi i}\int_{|\xi|=1}f\circ \varphi_x(\xi y)\frac{d\xi}{\xi^2}.$$

Thus equality (\ref{eq11}) becomes
\begin{eqnarray}
& \lefteqn{s_x^2\langle y, \overline{\nabla f(x)}\rangle=} \nonumber \\&
(s_x-1)\frac{1}{2\pi i}\int_{|\xi|=1}f\circ \varphi_x(\xi x)\frac{d\xi}{\xi^2}\frac{\langle y,x\rangle }{\|x\|^2}-s_x\frac{1}{2\pi i}\int_{|\xi|=1}f\circ \varphi_x(\xi y)\frac{d\xi}{\xi^2} \label{eq9}
\end{eqnarray}
\noindent and we conclude by taking $y=x$.
\end{proof}
\begin{rem} From (\ref{eq11}) we deduce the following identity that might be of independent interest \begin{equation} \label{eq12} s_x^2\nabla f(x)+ s_x\widetilde \nabla f(x)=(s_x-1)\frac{\langle \widetilde \nabla f(x), \bar x\rangle }{\|x\|^2} \bar x .\end{equation} \end{rem}
\begin{lema} \label{lemafinal} For every $0 <\delta<1,$ there exists $C_\delta>0$ such that
\be
|\langle y,\overline{\nabla f(x)}\rangle - \langle y', \overline{\nabla  f(x')}\rangle|\le C_\delta \|f\|_{\mathcal B}\big(\|x- x'\|+\frac{1-\delta}{2}\|y-y'\|\big)
\end{equation}
whenever $x,x'\in \delta B_E$ and $\|y\|\le 1,\|y'\|\le 1,$ and $f\in \mathcal B(B_E).$
\end{lema}
\begin{proof}
Let $\e=\frac{1-\delta}{2}.$ Since $ \max \{\|x+\e\xi y\|, \|x'+\e\xi y'\|: |\xi|=1\}\le \frac{1+\delta}{ 2},$
we conclude by taking $u=0$ in (\ref{eq4}) that
$$\rho_E(x+\e\xi y, x'+\e\xi y')\leq \frac{4(1+\delta)}{ 4+(1+\delta)^2}<1.$$
Since $\frac{1}{2} \log \frac{1+ r}{1-r} \le \frac{r}{1 -r}$ for all $0 < r <1,$ we have
$$\beta_E(x+\e\xi y, x'+\e\xi y')\le \frac{\rho_E(x+\e\xi y, x'+\e\xi y')}{ 1 - \frac{4(1+\delta)}{ 4+(1+\delta)^2}},$$
so  it follows that for some constant $C'_\delta$ depending only on $\delta,$
$$\beta_E(x+\e\xi y, x'+\e\xi y')\le C'_\delta \rho_E(x+\e\xi y, x'+\e\xi y').$$

Next, using Cauchy formula we have for  $x,x'\in \delta B_E$, $\|y\|\le 1,\|y'\|\le 1,$ \ba &&\langle y,\overline{\nabla f(x)}\rangle - \langle y', \overline{\nabla  f(x')}\rangle
=\frac{1}{\varepsilon}\frac{1}{2\pi i}\int_{|\xi|=1}f(x+\e\xi y)-f(x'+\e\xi y')\frac{d\xi}{\xi^2}.\ea

From this, Theorem \ref{teo met} and the equivalence of the semi-norms, we get that for some constant $C>0$
\ba
|\langle y,\overline{\nabla f(x)}\rangle - \langle y', \overline{\nabla  f(x')}\rangle|&\le&  \frac{1}{\varepsilon}\int_{0}^{2\pi}\Big|f(x+\e e^{it} y)-f(x'+\e e^{it} y')\Big|\frac{dt}{2\pi} \\
&\le&\frac{1}{\varepsilon} C\|f\|_{\B}\int_{0}^{2\pi}\beta_E(x+\e e^{it} y, x'+\e e^{it} y') \frac{dt}{2\pi}.\ea

 Applying (\ref{eq5}) we find a constant $C_\delta>0$ depending only on $\delta$ such that
\ba
|\langle y,\overline{\nabla f(x)}\rangle - \langle y', \overline{\nabla  f(x')}\rangle|\le
\frac{1}{\varepsilon}C\cdot C'_\delta\|f\|_{\B}\int_{0}^{2\pi}\rho_E(x+\e e^{it} y,x'+\e e^{it} y')\frac{dt}{2\pi}\\
\le  C_\delta\|f\|_{\B}\int_{0}^{2\pi}\|(x+\e e^{it} y)- (x'+\e e^{it} y')\|\frac{dt}{2\pi} \\\le
 C_\delta\|f\|_{\B}(\|x- x'\| +\e\|y-y'\|)= C_\delta\|f\|_{\B}\big(\|x- x'\| +\frac{1-\delta}{2}\|y-y'\|\big).\ea
\end{proof}
\begin{teo} \label{main} Let $\phi:B_E\to B_E$ be analytic. Assume that
\begin{flalign*}
(i) \quad \{\phi(z): \|\phi(z)\|\le \delta\} \mbox{ and } \{(1-\|z\|^2)\mathcal R\phi(z):\|\phi(z)\|\le \delta\} \mbox{ are relatively compact for all } 0<\delta<1,& &
\end{flalign*}
\begin{flalign*}
(ii) \quad \lim_{\|\phi(z)\|\to 1}\frac{(1- \|z\|^2) \|\mathcal R \phi(z)\|}{\sqrt{1- \|\phi(z)\|^2}}=0 &\;\mbox{ and} &  &
\end{flalign*}

\begin{flalign*}
(iii) \quad \lim_{\|\phi(z)\|\to 1}\frac{(1- \|z\|^2) |\langle  \phi(z), {\mathcal R\phi(z)}\rangle|}{1- \|\phi(z)\|^2}=0. & &
\end{flalign*}
Then $C_\phi:\mathcal B(B_E)\to \mathcal B(B_E)$ is a compact operator.
\end{teo}
\begin{proof} We are going to apply Lemma \ref{lema3}.
Let $(f_\alpha)$ be a bounded  net  in $\mathcal B(B_E)$  converging to zero uniformly on compact sets.
Recall that $$\mathcal R (f_\alpha\circ \phi)(z)=\langle \nabla f_\alpha(\phi(z)),\overline{\mathcal R\phi(z)}\rangle.$$

Let $\e>0$. By (ii) and (iii) there exists $\delta<1$ such that for $\|\phi(z)\|>\delta$ we have
$$(1-\|z\|^2)\frac{\sqrt{(1- \|\phi(z)\|^2)\|\mathcal R\phi(z)\|^2 + |\langle  \phi(z), {\mathcal R\phi(z)}\rangle|^2}}{1- \|\phi(z)\|^2} < \e$$
and hence  using Lemma \ref{lema2},  we have
\be \label{ineq1}
(1-\|z\|^2)|\mathcal R (f_\alpha\circ \phi)(z)|\le \sup_\alpha \|\widetilde \nabla f_\alpha\big(\phi(z)\big)\| \e\le \sup_\alpha \|f_\alpha\|_{inv}~\e.
\end{equation}
Denote $A_\delta=\{z\in B_E: \|\phi(z)\|\le \delta\}.$
For $z\in A_\delta$  we use formula (\ref{eq9}) obtained in the proof of Lemma \ref{lema tec} to have
\ba &&\langle \frac{\mathcal {R}\phi(z)}{2\|\mathcal R\phi(z)\|}, \overline{\nabla f(\phi(z))} \rangle=\\&=& \frac{1}{s_{\phi(z)}}\big(1-\frac{1}{s_{\phi(z)}}\big)\frac{1}{2\pi i}\int_{|\xi|=1}f\big(\varphi_{\phi(z)}(\xi \phi(z))\big)\frac{d\xi}{\xi^2}\frac{\langle \mathcal R\phi(z), \phi(z)\rangle}{{2\|\mathcal R\phi(z)\|}\|\phi(z)\|^2} \\
&-&  \frac{1}{s_{\phi(z)}}\frac{1}{2\pi i}\int_{|\xi|=1}f\big(\varphi_{\phi(z)}(\xi \frac{\mathcal R\phi(z)}{2\|\mathcal R\phi(z)\|})\big)\frac{d\xi}{\xi^2}.\ea

Hence for each $z\in A_\delta,$
\ba
&&(1-\|z\|^2)|\langle\nabla f_\alpha(\phi(z)), \overline{\mathcal R\phi(z)}\rangle |\le\\
& & \frac{(1-\|z\|^2)}{\|\phi(z)\|}\frac{1}{s_{\phi(z)}}(\frac{1}{s_{\phi(z)}}-1)\|\mathcal R\phi(z)\|
\int_{0}^{2\pi}|f_\alpha\big(\varphi_{\phi(z)}(e^{it} \phi(z))\big)|\frac{dt}{2\pi }\\
&+& \frac{2(1-\|z\|^2)}{s_{\phi(z)}}\|\mathcal R\phi(z)\|\int_{0}^{2\pi}\big|f_\alpha\big(\varphi_{\phi(z)}\big(e^{it}  \frac{\mathcal R\phi(z)}{2\|\mathcal R\phi(z)\|}\big)\big)\big|\frac{dt}{2\pi }.
\ea

 Bearing in mind (\ref{sch3}) in Lemma \ref{lema1} and that $\lim_{\epsilon\to 0} \frac{1}{\epsilon}(\frac{1}{\sqrt{1-\epsilon^2}}-1)= 0,$ there is $C>0$ such that for $\|\phi(z)\|\le \delta$ we have
$$\frac{(1-\|z\|^2)}{\|\phi(z)\|}\frac{1}{s_{\phi(z)}}(\frac{1}{s_{\phi(z)}}-1)\|\mathcal R\phi(z)\|
\le  \frac{2}{\|\phi(z)\|} \Big(\frac{1}{(1-\|\phi(z)\|^2)^{1/2}}-1\Big)
\le C.$$

In particular, for each $\delta<1$ there exists $C_\delta>0$ such that
$$(1-\|z\|^2)|\langle\nabla f_\alpha(\phi(z)), \overline{\mathcal R\phi(z)}\rangle |\le$$
$$C_\delta \Big(\int_{0}^{2\pi}|f_\alpha(\varphi_{\phi(z)}(e^{it} \phi(z)))\frac{dt}{2\pi }+
\int_{0}^{2\pi}|f_\alpha(\varphi_{\phi(z)}(e^{it} \frac{\mathcal R\phi(z)}{2\|\mathcal R\phi(z)\|}))|\frac{dt}{2\pi }\Big)$$
when  $\|\phi(z)\|\le \delta.$
Therefore, since $$\{ \varphi_{\phi(z)}(\xi \phi(z)):\xi\in \mathbb T \}\cup
\{ \varphi_{\phi(z)}(\xi  \frac{\mathcal R\phi(z)}{2\|\mathcal R\phi(z)\|}):\xi\in \mathbb T \}$$ is compact in $B_E$, we have for each $z\in A_\delta$ that
\be\label{final}
(1-\|z\|^2)|\langle\nabla f_\alpha(\phi(z)), \overline{\mathcal R\phi(z)}\rangle |\to 0.
\end{equation}

Now bearing in mind $(\ref{sl2})$ to observe that $s_y^2\|\mathcal R\phi(y)\|\le 1, $ we may use Lemma \ref{lemafinal} to have for each
$z, z'\in  A_\delta$
$$\Big|\langle\nabla f_\alpha(\phi(z)), s_z^2\overline{\mathcal R\phi(z)}\rangle-\langle\nabla f_\alpha(\phi(z')), s_{z'}^2\overline{\mathcal R\phi(z')}\rangle\Big|$$ $$ \le C_\delta \|f_\alpha\|_{\mathcal B}\Big(\|\phi(z)- \phi(z')\|+\frac{1-\delta}{2}(\|s_z^2\overline{\mathcal R\phi(z)}-s_{z'}^2\overline{\mathcal R\phi(z')}\|)\Big).$$

To finish the proof we use that  both $\phi(A_\delta)$  and  $\{(1-\|z\|^2)\mathcal R\phi(z):z\in A_\delta\}$ are relatively compact, thus also the set $\{\big(\varphi(z),(1-\|z\|^2)\mathcal R\phi(z)\big):z\in A_\delta\}\subset E\times E$ is relatively compact. So, given $\e>0$ there exists a finite family of points $\{z_k: 1\le k\le N\}\subset A_\delta$ such that
for each $z\in A_\delta$  there exists $z_k$ for which $\|\phi(z)-\phi(z_k)\|+\frac{1-\delta}{2}(\|s_z^2\overline{\mathcal R\phi(z)}-s_{z_k}^2\overline{\mathcal R\phi(z_k)}\|)<\e$.

$$\text{ Hence }~~~ \sup_{z\in A_\delta }|\langle \nabla f_\alpha(\phi(z)),s_{z}^2\overline{\mathcal R\phi(z)} \rangle|\le C'2\varepsilon + \max_{1\le k\le n} |\langle \nabla f_\alpha(\phi(z_k)), s_{z_k}^2\overline{\mathcal R\phi(z_k)} \rangle|.~~~~$$
The proof is then complete using (\ref{final}).
\end{proof}

\begin{cor} Assume that $\{\phi(z): \|\phi(z)\|\le \delta\}$ and $\{(1-\|z\|^2)\mathcal R\phi(z):\|\phi(z)\|\le \delta\}$ \mbox{ are relatively compact for all } $\delta<1$. Then $C_\phi:\mathcal B(B_E)\to \mathcal B(B_E)$ is a compact operator if and only if

$$(i) ~~~~~~~~~~\lim_{\|\phi(z)\|\to 1}\frac{(1- \|z\|^2) \|\mathcal R \phi(z)\|}{\sqrt{1- \|\phi(z)\|^2}}=0, \hbox{ and } $$

 $$(ii) ~~~~~~~~\lim_{\|\phi(z)\|\to 1}\frac{(1- \|z\|^2) |\langle  \phi(z),
 {\mathcal R\phi(z)}\rangle|}{1- \|\phi(z)\|^2}=0. $$
\end{cor}
\begin{cor} \label{t0} Assume  that $\|\phi\|_\infty<1.$
The composition operator $C_\phi$ is compact if $ \phi(B_E) $  is relatively compact.
 \end{cor}\begin{proof} It is enough to check that the set $ \{(1-\|z\|^2)\mathcal R\phi(z):z\in B_E\}$
 is relatively compact.  Lemma \ref{lema tec} applied to $\mu\circ\varphi$ for all $\mu\in E^*$  yields $(1-\|z\|^2)\mathcal R\phi(z)=\frac{-1}{2\pi i }\int_{|\xi|=1} \varphi(\varphi_x(\xi x))\frac{d\xi}{\xi^2}\Big.$ Hence $(1-\|z\|^2)\mathcal R\phi(z)$ belongs to the weak-closure of the balanced convex hull of the compact set  $\overline{\{\frac{1}{\xi^22\pi}\phi\big(B_E\big):|\xi|=1\}}\subset E$ that is also a compact set.\end{proof}
\begin{ex} \label{com1} Let $\{e_n\}$ be a sequence in the given basis $\{e_k\}.$ If $\{\phi_n\}$ is a sequence in $H^\infty(B_E)$  such that $\sum_{n=1}^\infty\|\phi_n\|^2_\infty<1,$
then the mapping $\phi(z):=\sum_n \phi_n(z) e_n$ yields a compact composition
 operator $C_\phi$  on $\mathcal B(B_E).$

 In particular for $\phi_n(z)=\prod_{j=n}^{2n}\langle z,e_j\rangle,$  $C_\phi$ is compact on $\mathcal B(B_E).$
\end{ex}
\begin{proof} Note that $\sup_{\|z\|<1}\|\phi(z)\|^2\le (\sum_{n=1}^\infty\sup_{\|z\|<1}|\phi_n(z)|^2)<1.$ Moreover, $\phi(B_E)$ is relatively compact since it
lies inside the Hilbert cube  $\mathcal{H}$ given by the sequence $(\|\phi_n\|_\infty).$
Now, apply Corollary \ref{t0}.

To verify the particular case we use the inequality between geometric and arithmetic means, namely $$|\phi_n(z)|=\prod_{j=n}^{2n}|\langle z,e_j\rangle|\le \Big(\frac{1}{n+1}\sum_{j=n}^{2n}|\langle z,e_j\rangle|\Big)^{n+1}\le (n+1)^{-\frac{n+1}{2}}\|z\|,$$
which produces the estimate  $\sum_{n=1}^\infty\|\phi_n\|_\infty^2 \le \sum_{n=1}^\infty (n+1)^{-(n+1)}<1.$
\end{proof}

\medskip

Next, we introduce a class of symbols $\varphi$ that allows a characterization of the  compactness of $C_\varphi.$ We say that the analytic mapping $\phi:B_E\to B_E$ belongs to   $\mathcal{B}_0(B_E,B_E)$
if \begin{equation}\label{bloch0} \lim_{\|z\|\to 1} (1-\|z\|^2)\|\mathcal{R}\phi(z)\|=0.\end{equation}
 In particular any map with bounded radial derivative satisfies (\ref{bloch0}).
It is easy to produce examples of maps in $\mathcal B_0(B_E,B_E):$
\begin{prop} Let $\{e_n\}$ be a sequence in the given basis $\{e_k\}.$ If $\{\phi_n\}_n\subset\mathcal B(B_E)$ is  such that $$\lim_{\|z\|\to 1} (1-\|z\|^2)|\mathcal{R}\phi_n(z)|=0 ~~\hbox{for all }n\in\mathbb{N} ~~\hbox{ and }~~
 \sum_{n=1}^\infty \|\phi_n\|^2_{\mathcal B(B_E)}<\infty,$$ then $\phi(z)=\sum_{n=1}^\infty \phi_n(z)e_n\in  \mathcal B_0(B_E,B_E)$.
\end{prop}
\begin{proof}Given $\e>0$ there exist $N\in \N$ and $0<\delta_j<1$ for $j=1, \cdots, N$ such that
$$(1-\|z\|^2)^2\|\mathcal R\phi(z)\|^2\le \sum_{n=1}^N (1-\|z\|^2)^2|\mathcal R\phi_n(z)|^2+ \e^2/2$$
and
$$(1-\|z\|^2)|\mathcal R\phi_j(z)|<\e/\sqrt{2N}, \quad \|z\|>\delta_j, \quad j=1,\cdots, N$$
Hence if $\|z\|>\max
_{1\le j\le N}\{\delta_j\},$ then $(1-\|z\|^2)\|\mathcal R\phi(z)\|<\e$.
\end{proof}
\begin{prop} \label{b0} Let $\varphi\in \mathcal{B}_0(B_E,B_E)$ with $\varphi(0)=0.$ Then $$(i)~~~\limsup_{\|z\|\to 1}\frac{(1- \|z\|^2) \|\mathcal R \phi(z)\|}{\sqrt{1- \|\phi(z)\|^2}}=\limsup_{\|\phi(z)\|\to 1}\frac{(1- \|z\|^2) \|\mathcal R \phi(z)\|}{\sqrt{1- \|\phi(z)\|^2}},  $$

 $$(ii)~~~\limsup_{\|z\|\to 1}\frac{(1- \|z\|^2) |\langle  \phi(z), {\mathcal R\phi(z)}\rangle|}{1- \|\phi(z)\|^2}= \limsup_{\|\phi(z)\|\to 1}\frac{(1- \|z\|^2) |\langle  \phi(z), {\mathcal R\phi(z)}\rangle|}{1- \|\phi(z)\|^2}.$$
\end{prop}
\begin{proof}
In case $\|\phi\|_\infty <1,$ both right hand limits are null and both
left hand limits vanish according to the assumption.

Since $\|\phi(z)\|\le\|z\|$ by Lemma \ref{lema1}, the limits on the right hand side
are not greater than those on the left hand side. Now, in case $\|\phi\|_\infty =1,$
there is a sequence $(z_n)\subset B_E$ such that $\|z_n\|\to 1$ and $ \limsup_{\|z\|\to 1}\frac{(1- \|z\|^2) \|\mathcal R \phi(z)\|}{\sqrt{1- \|\phi(z)\|^2}}= \lim_n\frac{(1- \|z_n\|^2) \|\mathcal R \phi(z_n)\|}{\sqrt{1- \|\phi(z_n)\|^2}} .$ From the bounded sequence $(\|\phi(z_n)\|)$ we get a convergent subsequence that we denote the same. If  $\lim_n \|\phi(z_n)\|=1,$ we have $ \limsup_{\|\phi(z)\|\to 1}\frac{(1- \|z\|^2) \|\mathcal R \phi(z)\|}{\sqrt{1- \|\phi(z)\|^2}}\ge \lim_n\frac{(1- \|z_n\|^2) \|\mathcal R \phi(z_n)\|}{\sqrt{1- \|\phi(z_n)\|^2}} $ that leads to the equality $(i),$ while if $\lim_n \|\phi(z_n)\|<1,$ then $ \limsup_{\|z\|\to 1}\frac{(1- \|z\|^2) \|\mathcal R \phi(z)\|}{\sqrt{1- \|\phi(z)\|^2}}=0,$ so $(i)$ holds as well. The analogous argument for $(ii).$ \end{proof}
\smallskip

In the following result we replace condition (i) in Theorem \ref{main} by the weaker one given by (\ref{c0}) and  conditions (\ref{c1}) and (\ref{c2}) by the stronger ones given by taking $\lim_{\|\phi(z)\| \to 1} $  instead of $\lim_{\|z\| \to 1} .$ Since the proof  follows the same arguments as  in Theorem \ref{main}, it will only be  sketched.
\begin{prop} Let   $\phi:B_E\to B_E$ be analytic with $\phi(0)=0$. If  $\phi$ satisfies (\ref{c0}), \begin{equation}\label{c1'}
 \quad \lim_{\|z\|\to 1}\frac{(1- \|z\|^2) \|\mathcal R \phi(z)\|}{\sqrt{1- \|\phi(z)\|^2}}=0, \text{ and }
\end{equation}

\begin{equation} \label{c11}
\quad \lim_{\|z\|\to 1}\frac{(1- \|z\|^2) |\langle  \phi(z), {\mathcal R\phi(z)}\rangle|}{1- \|\phi(z)\|^2}=0,
\end{equation} then $C_\phi$ is compact on $\B.$
\end{prop}
\begin{proof} By Lemma \ref{lema1}, we have $\|\phi(z)\|\le\|z\|.$ The analogous estimate to (\ref{ineq1}) holds for $\|z\|>\delta.$

In the remaining case $\|z\|\le\delta,$ also $\|\phi(z)\|\le\delta$ so the estimates in the proof of Theorem \ref{main} hold, that is, if $\|z\|\le\delta,$ then \be\label{final1}
(1-\|z\|^2)|\langle\nabla f_\alpha(\phi(z)), \overline{\mathcal R\phi(z)}\rangle |\to 0.
\end{equation}
Now the final argument in the proof of Theorem \ref{main} relies on the relative compactness of $\phi(\{\|z\|\le\delta\}$  that holds by assumption and that of $\{\mathcal R\phi(z):\|z\|\le\delta\},$ which follows from the Cauchy formula: Indeed,
$\mathcal R\phi(z)=\varphi'(z)(z)=\frac{1}{2\pi i}\int_{|\lambda|=r}\frac{\varphi(z+\lambda z)}{\lambda^2}d\lambda,$ for $0<r<1$ such that $\delta+r<1.$ Therefore, $\mathcal R\phi(z)$ belongs to the weak-closure of the balanced convex hull of the compact set  $\overline{\{\mu\phi\big((\delta+r)B_E\big):|\mu|=r^{-2}\}}\subset E$ that is also a compact set.
\end{proof}\smallskip

Let us mention that (\ref{c1'}) implies that $\phi\in \mathcal{B}_0(B_E,B_E)$ and that combining  the necessary condition obtained in Theorem \ref{sufi1} and Proposition \ref{b0} we get the following:
\begin{cor} Let $\varphi\in \mathcal{B}_0(B_E,B_E)$ with $\phi(0)=0$. Then $C_\varphi$ is compact in $\B$ if and only if
$\phi$ satisfies (\ref{c0}), (\ref{c1'}) and (\ref{c11}).
\end{cor}
\section{Examples}
In this section we provide a number of examples to discuss the relations among the various conditions we have found above.
\begin{ex}  Consider $(\xi_n)\subset B_E$ such that   \begin{equation}\label{a1}\sup_{\|z\|\le 1} \sum_n |\langle z, \xi_n\rangle|^2\le 1.\end{equation}
Define $\phi_n(z)= \langle z, \xi_n\rangle$ and $\phi(z)=\sum_n \phi_n(z)e_n,$ where $\{e_n\}$ is an orthonormal sequence in $E.$

Then $\phi$ satisfies (\ref{c1'}). In particular $\phi \in \mathcal B_0(B_E, B_E)$.

Moreover if $(\xi_n)$ is an orthogonal system we have that

 (i) $\phi$ satisfies (\ref{c2}) whenever $\sup_{n}\|\xi_n\|<1,$

 (ii) $\phi$ fails (\ref{c2}) whenever there exists $n_0$ with $\|\xi_{n_0}\|=1,$

(iii)   $\phi(B_E)$ is relatively compact whenever $\sum_{n}\|\xi_n\|^2<\infty,$ and

(iv)  $\phi$ fails (\ref{c0}) whenever $\limsup_{n\to \infty}\|\xi_n\|>0$.
\end{ex}
\begin{proof} Assumption (\ref{a1}) guarantees that $\phi$ is analytic and maps $B_E$ to $B_E$. Since $\phi(0)=0$, by Lemma \ref{lema1}, we have $\|\phi(z)\|\le \|z\|$ for any $z \in B_E$. Notice that $\mathcal R\phi(z)=\sum_n \mathcal{R}\varphi_n(z)e_n=\phi(z)$ and using that $\alpha \mapsto \frac{\alpha}{\sqrt{1-\alpha^2}}$ is increasing for $0<\alpha<1,$   we have
$$\frac{(1- \|z\|^2) \|\mathcal R \phi(z)\|}{\sqrt{1- \|\phi(z)\|^2}}=\frac{(1- \|z\|^2) \| \phi(z)\|}{\sqrt{1- \|\phi(z)\|^2}}\le \sqrt{1- \|z\|^2}.$$ In particular $\phi$ satisfies (\ref{c1}).

Since $\langle  \phi(z), {\mathcal R\phi(z)}\rangle =\|\phi(z)\|^2, $ we get that $\phi$ satisfies (\ref{c2}) if and only if $\lim_{\|\phi(z)\|\to 1} \frac{1- \|z\|^2}{1- \|\phi(z)\|^2}=0$.

Assume now that $(\xi_n)$ is an orthogonal system. Hence
$$\|\phi(z)\|^2= \sum_n |\langle z,\xi_n\rangle|^2\le \sup_n\|\xi_n\|^2 \sum_n |\langle z,\frac{\xi_n}{\|\xi_n\|}\rangle|^2\le \sup_n \|\xi_n\|^2 \|z\|^2.$$
Assuming $\sup_n \|\xi_n\|^2 <1$ we have $\phi(B_E)\subset \delta B_E$ for some $\delta<1$ and (\ref{c2}) trivially holds which shows (i).

Assume now that $\|\xi_{n_0}\|=1$.
Selecting $z=\lambda \xi_{n_0}$ we have that $\phi(z)= \lambda  e_{n_0}$ and
$$\lim_{\|\phi(z)\| \to 1} \frac{1- \|z\|^2 }{1- \|\phi(z)\|^2}\|\phi(z)\|^2=1.$$
This gives (ii).

Now (iii) follows using that $|\phi_n(z)| \le \|\xi_n\|$ for each $n$. Hence $\phi(B_E)$  is contained in the Hilbert cube given by the sequence $(\|\xi_n\|).$

Finally, to show (iv),  assume $\limsup_{n\to \infty}\|\xi_n\|>0$. Hence there exist $\e>0$ and indices $m_n$ such that $\|\xi_{m_n}\|\ge \e$. For each $0<\delta<1$ we have $\phi(\delta \frac{\xi_n}{\|\xi_n\|})=\delta \|\xi_n\|e_n$. Hence $\{ \delta \|\xi_{m_n}\|e_{m_n}:n\in \N\}\subset \phi(\delta B_E)$ which gives that $\phi(\delta B_E)$ is not relatively compact in $B_E$.
\end{proof}
\smallskip

In  \cite{GMS}
it was shown that $\phi(z)=\sum_{n=1}^\infty z_n^n e_n$
satisfies (\ref{c0}).  Such $\varphi$  is a particular choice in the following example.
\begin{ex} Let $\{e_k\}$ be an orthonormal sequence in $E.$
Let  $F_k:\D\to \D$ be a sequence of analytic functions such that $F_k(0)=0.$  Define
$$\phi(z):=\sum_{k=1}^\infty F_k(\langle z,e_k\rangle)e_k.$$

(i) If $\|F_k\|_\infty<1$ for all $k\in \N,$ then $\phi$ satisfies (\ref{c4}).

(ii) If $F_k\in \mathcal B_0,$ the little Bloch space, and $\|F_k\|_\infty<1$ for all $k\in \N,$ then $\phi$ satisfies (\ref{c3'}).

(iii) If  there exists $n_0\in \N$ such that $C_{F_{n_0}}$ is non-compact on $\mathcal B$, then $\phi$ fails  (\ref{c2}).

(iv) If $\sup_k\|F_k\|_\infty<1,$ then $\phi$ satisfies $\phi(B_E)\subset \delta B_E$ for some $0<\delta<1$. In particular $\varphi$ satisfies (\ref{c1}) and (\ref{c2}).

(v) If $\sum_k\|F_k\|^2_\infty<\infty,$ then  $\phi(B_E)$ is relatively compact in $B_E$.

(vi) If $\sum_k\|F_k\|^2_{\mathcal B}<\infty,$ then  $\phi$ satisfies (\ref{c0}).

\end{ex}

\begin{proof}
Notice that since $|F_k(\lambda)|\le |\lambda|$ we have that $\phi$ maps $B_E$ into $B_E$. Actually one has
$$\|\phi(z)\|^2=\sum_{k=1}^\infty |F_k(\langle z,e_k\rangle)|^2\le \sum_{k=1}^\infty \|F_k\|^2_\infty|\langle z,e_k\rangle|^2\le \|z\|^2~ ~~\text{ and further,}$$

\begin{equation}\label{e1}\|\phi(z)\|\le (\sup_k \|F_k\|_\infty) \|z\|.\end{equation}
 Since
$
\phi'(z)(u)=\sum_{k=1}^\infty F'_k(\langle z,e_k\rangle)\langle u,e_k\rangle e_k,$ then
$$\mathcal R\phi(z)=\sum_{k=1}^\infty F'_k(\langle z,e_k\rangle)\langle z,e_k\rangle e_k, \text { so }
$$
$$
\langle\phi(z),\mathcal R\phi(z)\rangle=\sum_{k=1}^\infty F_k(\langle z,e_k\rangle)\overline{F'_k(\langle z,e_k\rangle)\langle z,e_k\rangle}.
$$

Statement (i) follows since $\phi_{k,l}(\lambda)= F_k(\lambda)\delta_{k,l}$ and $\|F_k\|_\infty<1$ implies compactness of $C_{F_k}$.
  $$\text{ To verify (ii) notice that }\;\;\frac{|\mathcal R\phi_n(z)|}{1-|\phi_n(z)|^2}=\frac{|F'_n(\langle z,e_n\rangle)\|\langle z,e_n\rangle|}{1- |F_n(\langle z,e_n\rangle)|^2}\;\;\;\text{
from where we conclude } $$
$$\frac{(1-\|z\|^2)|\mathcal R\phi_{n}(z)|}{1-|\phi_n(z)|^2}\le \frac{(1-|\langle z,e_n\rangle|^2)|F'_n(\langle z,e_n\rangle)|}{1-\|F_n\|_\infty^2}$$
that shows (\ref{c3'}).

Concerning (iii): since $C_{F_{n_0}}$ is non-compact then by Theorem 2 in \cite{MM} there exists $(\lambda_n) \subset \D$ for which $|F_{n_0}(\lambda_n)|\to 1$ (in particular $|\lambda_n|\to 1$) and
$$\lim_n \frac{(1- |\lambda_n|^2)|F'_{n_0}(\lambda_n)|}{1- |F_{n_0}(\lambda_n)|^2}\neq 0.$$
$$\hbox{Selecting the sequence }\xi_n=\lambda_ne_{n_0}, \text{ we  have } \|\phi(\xi_n)\|^2=|F_{n_0}(\lambda_n)|^2,  ~~~~\|\mathcal R\phi(\xi_n)\|=|F'_{n_0}(\lambda_n)\|\lambda_n|$$  and
 $\langle \mathcal R\phi(\xi_n),\phi(\xi_n)\rangle = \overline{F_{n_0}(\lambda_n)}F'_{n_0}(\lambda_n)\lambda_n.$ Therefore  $\phi$ fails   (\ref{c2}).
\smallskip

 To check (iv), choose $\delta=\sup_k \|F_k\|_\infty$ and  use (\ref{e1}).

 Since $\phi(B_E)$  is contained in the Hilbert cube given by the sequence $(\|F_k\|_\infty)$ it is relatively compact. Thus  (v) holds.

 Finally to show (vi) we use the estimate for analytic functions $F:\D\to \D$ with $F(0)=0$ given by
 $|F(\lambda)|\le \|F\|_{\mathcal B} \beta(0,\lambda)$ to obtain that $\phi(\delta B_E)$ is  contained in the Hilbert cube given by the sequence $(\|F_k\|_{\mathcal B} \beta(0,\delta)).$ This gives (\ref{c0}).
\end{proof}
\begin{ex} \label{nc2} Let $\{e_k\}$ be an orthonormal sequence in $E.$ Let us consider $\phi(z)=\sum_k \phi_k(z)e_k$ where
\begin{equation}\phi_k(z)=\langle z,e_k\rangle^k.\end{equation}
Then $\phi$ satisfies (\ref{c0}) and fails (\ref{c3}). In particular $C_\phi$ is non-compact on $\mathcal B(B_E)$.
\end{ex}
\begin{proof} Notice that $\phi(z)\in B_E$ for each $z\in B_E$ because
$$\sum_{k=1}^\infty |\phi_k(z)|^2 \le \sum_{k=1}^\infty  |\langle z,e_k\rangle|^{2}\le \|z\|^2.$$
It is clear that $\mathcal R\phi_k(z)=k\phi_k(z)$.

To show (\ref{c0}) just observe that $\sup_{\|z\|\le \delta} |\phi_k(z)|\le \delta^k.$ Denote
$$A_k=\sup_{z\in B_E} \frac{(1-\|z\|^2)|\mathcal R\phi_k(z)|}{1-|\phi_k(z)|^2}.$$
\ba \text{Let}~~z=\lambda e_k \text{ and estimate }
A_k\ge \sup_{0<\lambda<1}
\frac{(1-\lambda^2)k\lambda^k}{1-\lambda^{2k}}\ge \sup_{k}(1-\frac{1}{k})^{\frac{k}{2}}>0.
\ea
\end{proof}
\begin{ex} \label{nc1} Let $\{e_k\}$ be an orthonormal sequence in $E.$ Let $(n_k)_{k\in \N}$ be an increasing sequence of natural numbers with $n_0=0$ and define $\phi(z)=\sum_k \phi_k(z)e_k$
and $\psi(z)=\sum_k \psi_k(z)e_k$ where
\begin{equation}\phi_k(z)=\sum_{j=n_{k-1}+1}^{n_{k}} z_j^{2k}
\hbox{ and } \psi_k(z)=\left( \sum_{j=n_{k-1}+1}^{n_{k}} z_j^{2} \right)^k.\end{equation}
Then  $\phi$ and $\psi$ satisfy (\ref{c0}) but fail (\ref{c4}). Hence $C_\phi$ and $C_\psi$ are non-compact on $\mathcal B(B_E)$.
\end{ex}
\begin{proof}
Notice that $\phi(z), \psi(z) \in B_E$ for each $z\in B_E$ because
$$\max\{ |\phi_k(z)|^2, |\psi_k(z)|^2\}\le  (\sum_{j=n_{k-1}+1}^{n_{k}} |z_j|^{2}\Big)^{k}\le \sum_{j=n_{k-1}+1}^{n_{k}} |z_j|^{2}.$$

Condition (\ref{c0}) follows from the estimate~~~~ $\max\{ |\phi_k(z)|^2, |\psi_k(z)|^2\}\le \|z\|^{2k}.$

It is immediate to see that $\mathcal R\phi_k(z)=2k\phi_k(z)$  and $\mathcal R\psi_k(z)=2k\psi_k(z)$ and, for each $k,m\in \N$ we have
$$\psi_{k,m}(\lambda)=\phi_{k,m}(\lambda)=  \lambda^{2k}, \quad n_k\le m\le n_{k+1}$$
and $\psi_{k,m}=\phi_{k,m}=0$ otherwise.

We see that $C_{\phi_{k,m}}$ is non-compact on $\mathcal B$ because
$$\lim_{|\lambda^{2k}|\to 1}\frac{(1- |\lambda|^2) 2k |\lambda|^{2k-1}}{1- |\lambda|^{4k}}\ne 0$$ due to the estimate
$1- |\lambda|^{4k}\le 2k(1- |\lambda|^{2}), \quad |\lambda|<1.$
\end{proof}

%
%


\small
\begin{thebibliography}{Coif}
\bibitem{BGM} O. Blasco, P. Galindo and A. Miralles, \textit{Bloch functions on the unit ball of an infinite dimensional Hilbert space}, J. Func. Anal. {\bf 267} (2014),  1188--1204.
\bibitem{BLT} O. Blasco, M. Lindstr\"om and J. Taskinen,  \textit{Bloch-to-BMOA compositions in several complex variables}, Complex Var. Theory Appl. {\bf 50} (2005), no. 14, 1061--1080.
\bibitem{C} S. B. Chae,  \textit{Holomorphy and calculus in normed spaces. With an appendix by Angus E. Taylor}, Monographs and Textbooks in Pure and Applied Mathematics, 92. Marcel Dekker, Inc., New York, 1985. xii+421 pp.
\bibitem{D} J. Dai, \textit{Compact composition operators on the Bloch space of the unit ball}
J. Math. Anal. Appl. {\bf 386} (2012), 294--299.
\bibitem{GMS} D. Garc\'{i}a, M. Maestre and P. Sevilla-Peris, \textit{Composition operators between weighted spaces of holomorphic functions on Banach spaces}, Ann. Acad. Sci. Fenn. Math. {\bf 29} (2004), 81--98.
\bibitem{GR} K. Goebel and S. Reich,  \textit{Convexity, Hyperbolic Geometry, and Nonexpansive Mappings}, Marcel Dekker, Inc., New York and Basel (1984).
\bibitem{MM} K. Madigan and A. Matheson, \textit{Compact composition operators on the Bloch space},
Trans. Amer. Math. Soc. {\bf 347} (1995) 2679--2687.
\bibitem{T80}
R. M. Timoney, Bloch functions in several complex variables I, \emph{Bull. London Math. Soc.}, \textbf{12} (1980) 241--267.
\bibitem{Z} K. Zhu, \textit{Spaces of holomorphic functions in the unit ball},
Grad. Texts in Math. {\bf 226}, Springer Verlag 2005.
\bibitem{Z2}
K. Zhu, \emph{Operator theory in function spaces}, Mathematical Surveys and Monographs \textbf{138}, American Mathematical Society, Providence, RI (2007).
\end{thebibliography}
\end{document}